\documentclass[11pt,reqno]{amsart}
\usepackage{amsfonts,amssymb,amsthm}
\usepackage{amsmath,amscd}
\usepackage{pstricks}
\usepackage{pstricks,pst-node}
\usepackage{mathrsfs}
\usepackage[all]{xy}

\DeclareMathAlphabet{\mathpzc}{OT1}{pzc}{m}{it}

\renewcommand{\subsection}[1]{\vspace{.18in}
\par\noindent\addtocounter{subsection}{1}
\setcounter{equation}{0}{\bf\thesubsection.\hspace{5pt}#1}}

\theoremstyle{definition}

\theoremstyle{plain}
\newtheorem{Prop}[subsection]{Proposition}
\newtheorem{Thm}[subsection]{Theorem}
\newtheorem{MThm}[subsection]{Main Theorem}
\newtheorem{Not}[subsection]{Further Notations}
\newtheorem{Lem}[subsection]{Lemma}
\newtheorem{Coro}[subsection]{Corollary}
\numberwithin{equation}{subsection}

\newcommand{\su}{{}^{}}
\def\su#1{^{#1}}

\newcommand{\tA}{{}^t\!A}

\newcommand{\bin}{\bigcup}

\newcommand{\han}{\subseteq}

\newcommand{\bsq}{{\boldsymbol{q}}}

\newcommand{\lan}{\langle}
\newcommand{\ran}{\rangle}
\newcommand{\dleb}{\left[\!\!\left[}
\newcommand{\leb}{\left[}
\newcommand{\drib}{\right]\!\!\right]}
\newcommand{\rib}{\right]}
\newcommand{\bbl}{\big[}
\newcommand{\bbr}{\big]}
\newcommand{\dbbl}{\big[\!\!\big[}
\newcommand{\dbbr}{\big]\!\!\big]}

\def\dblr#1{[\![#1]\!]}
\def\lr#1{\langle #1\rangle}

\def\ggp#1#2{\left[\kern-3.2pt\left[{#1\atop #2}\right]\kern-3.2pt\right]}

\def\fS{{\frak S}}

\def\fB{{\frak B}}

\def\fka{{\frak a}}

\newcommand{\msY}{\mathpzc Y}
\newcommand{\msZ}{\mathpzc Z}

\newcommand{\msD}{\mathscr D}\newcommand{\afmsD}{{\mathscr D}^\vtg}

\newcommand{\affSr}{{\fS_{\vtg,r}}}

\newcommand{\afHr}{{\sH_\vtg(r)}}

\newcommand{\afbfHr}{{\boldsymbol{\mathcal H}_\vtg(r)}}

\def\sD{{\mathcal D}}

\def\sH{{\mathcal H}}

\def\sN{{\mathcal N}}

\def\sS{{\mathcal S}}

\def\sU{{\mathcal U}}
\def\sV{{\mathcal V}}
\def\sW{{\mathcal W}}

\def\sY{{\mathcal Y}}
\def\sZ{{\mathcal Z}}

\newcommand{\vtg}{{\!\vartriangle\!}}
\newcommand{\Hall}{{{\mathfrak H}_\vtg(n)}}

\newcommand{\bfHall}{{\boldsymbol{\mathfrak H}_\vtg(n)}}

\newcommand{\dbfHa}{{\boldsymbol{\mathfrak D}_\vtg}(n)}
\newcommand{\dbfHap}{{\boldsymbol{\mathfrak D}^+_\vtg}(n)}
\newcommand{\dbfHam}{{\boldsymbol{\mathfrak D}^-_\vtg}(n)}
\newcommand{\dbfHaz}{{\boldsymbol{\mathfrak D}^0_\vtg}(n)}

\def\field{{\mathbb F}}

\newcommand{\mbnn}{\mathbb N^{n}}

\newcommand{\mnmod}{\!\!\!\mod\!}

\newcommand{\tri}{\triangle(n)}

\newcommand{\afgl}{\widehat{\frak{gl}}_n}

\newcommand{\afE}{E^\vartriangle}

\newcommand{\cycn}{{{n}}}
\newcommand{\afSr}{{\mathcal S}_{\vtg}(\cycn,r)}

\newcommand{\afbfSr}{{\boldsymbol{\mathcal S}}_\vtg(\cycn,r)}
\newcommand{\afbfS}{{\boldsymbol{\mathcal S}}_\vtg}

\newcommand{\afbse}{\boldsymbol e^\vartriangle}
\newcommand{\afmbnn}{\mathbb N_\vtg^{\cycn}}
\newcommand{\afmbzn}{\mathbb Z_\vtg^{\cycn}}

\newcommand{\afLa}{\Lambda_\vtg}
\newcommand{\afLanr}{\Lambda_\vtg(\cycn,r)}

\newcommand{\afThn}{\Theta_\vtg(\cycn)}

\newcommand{\afThnpm}{\Theta_\vtg^\pm(\cycn)}
\newcommand{\afThnp}{\Theta_\vtg^+(\cycn)}
\newcommand{\afThnm}{\Theta_\vtg^-(\cycn)}

\newcommand{\afThnr}{\Theta_\vtg(\cycn,r)}

\newcommand{\afMnz}{M_{\vtg,\cycn}(\mathbb Z)}
\newcommand{\afMnc}{M_{\vtg,\cycn}(\mathbb C)}
\newcommand{\afMnn}{M_{\vtg,\cycn}(\mathbb N)}

\newcommand{\afbfVn}{\boldsymbol{\sV}_\vtg(n)}
\newcommand{\afbfVnp}{\boldsymbol{\sV}_\vtg^+(n)}

\newcommand{\afbfVnz}{\boldsymbol{\sV}_\vtg^0(n)}

\def\leq{\leqslant}\def\geq{\geqslant}
\def\le{\leqslant}\def\ge{\geqslant}

\newcommand{\dt}{\bfj}

\def\de{{\delta}}

\newcommand{\og}{\omega}

\newcommand{\vi}{\varphi}

\newcommand{\up}{\boldsymbol{v}}

 \newcommand{\al}{\alpha}
 \newcommand{\bt}{\beta}
 \newcommand{\h}{\widehat}
 \newcommand{\ti}{\widetilde}
\newcommand{\zr}{\zeta_r}

\newcommand{\sg}{\sigma}

\def\th{\theta}

\newcommand{\p}{\prec}

\newcommand{\bop}{\bigoplus}

\newcommand{\ot}{\otimes}

\newcommand{\bfl}{\mathbf{0}}

\newcommand{\ol}{\overline}
\newcommand{\lra}{\longrightarrow}
\newcommand{\ra}{\rightarrow}
 \newcommand{\la}{{\lambda}}

 \newcommand{\La}{\Lambda}

 \newcommand{\mbn}{\mathbb N}
 \newcommand{\mbq}{\mathbb Q}
 \newcommand{\mbc}{\mathbb C}
 \newcommand{\mbz}{\mathbb Z}
 
 \newcommand{\bfi}{{\mathbf{i}}}
  \newcommand{\bfd}{{\mathbf{d}}}
 
 \newcommand{\bfj}{{\mathbf{j}}}

\newcommand{\bfU}{{\mathbf{U}}}

\newcommand{\ga}{{\gamma}}

\newcommand{\End}{\operatorname{End}}

\newcommand{\spann}{\operatorname{span}}
\newcommand{\diag}{\operatorname{diag}}

\def\ro{\text{\rm ro}}
\def\co{\text{\rm co}}

\newcommand{\ul}{\underline}

\def\afsygr{{\fS_{\vtg,r}}}

\def\fkH{\boldsymbol{\mathfrak H}}

\begin{document}
\title[quantum affine $\frak{gl}_n$]{Quantum affine $\frak{gl}_n$ via Hecke algebras}

\author{Jie Du}
\address{School of Mathematics and Statistics, University of New South Wales,
Sydney 2052, Australia.} \email{j.du@unsw.edu.au}
\author{Qiang Fu$^\dagger$}

\address{Department of Mathematics, Tongji University, Shanghai, 200092, China.}
\email{q.fu@tongji.edu.cn}

\date{\today}

\thanks{$^\dagger$Corresponding author.}
\thanks{Supported by the National Natural Science Foundation
of China, the Program NCET, Fok Ying Tung Education Foundation,
 the Fundamental Research Funds for the Central Universities, and the Australian Research Council DP120101436}

\begin{abstract} We use the Hecke algebras of affine symmetric groups and their associated Schur algebras to construct a new algebra through a basis, and a set of generators and explicit multiplication formulas of basis elements by generators. We prove that this algebra is isomorphic to the quantum enveloping algebra of the loop algebra of $\mathfrak {gl}_n$. Though this construction is motivated by the work \cite{BLM} by Beilinson--Lusztig--MacPherson for quantum $\frak{gl}_n$,  our approach is purely algebraic and combinatorial, independent of the geometric method which seems to work only for quantum $\mathfrak{gl}_n$ and quantum affine $\mathfrak{sl}_n$. As an application, we discover a presentation of the Ringel--Hall algebra of a cyclic quiver by semisimple generators and their multiplications by the defining basis elements. 
\end{abstract}
 \sloppy \maketitle
\section{Introduction}
The quantum enveloping algebra of the loop algebra of $\mathfrak {gl}_n$, or simply quantum affine $\mathfrak {gl}_n$, has two usual definitions, the $R$-matrix one and the Drinfled one, known as Drinfeld's new realisation. Both are presented by generators and relations (see, e.g., \cite[\S2.3]{FM} and the references therein). In \cite[2.3.1,2.5.3]{DDF}, a third presentation is given via the double Ringle--Hall algebra. In this presentation, the Ringel--Hall algebra of the cyclic quiver and its opposite algebra become the $\pm$-part of quantum affine $\mathfrak {gl}_n$. Thus, with this construction, one may consider semisimple or indecomposable generators for quantum affine $\mathfrak {gl}_n$, defined by the semisimple or indecomposable representations of the quiver; see \cite[\S1.4]{DDF}.
In particular, one sees easily the fact that the subalgebra generated by simple generators is a {\it proper} subalgebra. This subalgebra is isomorphic to the quantum affine $\mathfrak {sl}_n$.

The double Ringel--Hall algebra construction of quantum affine $\mathfrak {gl}_n$ is an affine generalisation of a similar construction for a quantum enveloping algebra $\bfU$ of a finite type quiver via a Ringel--Hall algebra which, as the positive or negative  part of $\bfU$, is spanned by the basis of isoclasses of representations of the quiver and whose multiplication is defined by Hall polynomials, see \cite{R90,R932,X97}. 
However, there is another construction for quantum $\mathfrak {gl}_n$  by Beilinson, Lusztig and MacPherson \cite[5.7]{BLM}, which directly displays a basis for the {\it entire} quantum enveloping algebra $\bfU(\mathfrak{gl}_n)$ and displays the multiplication rules by explicit formulas of basis elements by generators. This construction is geometric in nature and has been {\it partially} generalised to the affine case (more precisely, to affine $\mathfrak {sl}_n$) in \cite{GV,Lu99, VV99}. 

BLM's geometric approach uses the definition of quantum Schur algebras as the convolution algebras of (partial) flag varieties over finite fields and then, by a process of ``quantumization'', to get a construction over the polynomials ring. 
Progress on generalising BLM's work to the affine case via an algebraic approach has been made in the works \cite{DF09,DDF,Fu}. In particular, a realisation conjecture \cite[5.5(2)]{DF09} for quantum affine $\frak{gl}_n$ was formulated and was proved in the classical ($v=1$) case in \cite[Ch. 6]{DDF}. We will prove this conjecture in this paper.
 We will use directly the definition of quantum Schur algebras as endomorphism algebras of certain $\up$-permutation modules over the affine Hecke algebra which has a basis indexed by certain double cosets
of the affine symmetric group. The double cosets associated with semisimple representations will play a key role in the establishment of multiplication rules of BLM type basis elements by semisimple generators.

It should be pointed out that a recent work by Bridgeland constructs quantum enveloping algebras via Hall algebras of complexes and a complete realisation \cite[Th.~4.9]{Bri} is obtained for simply-laced finite types; see
\cite{SY} for the general (finite type) case. We obtain here a complete realisation for quantum affine $\mathfrak{gl}_n$.

We now describe the main result of the paper. For a positive integer $n$, let
$\afgl:=\afMnc$  be the loop algebra of $\mathfrak{gl}_n(\mbc)$ consisitng of all matrices
$A=(a_{i,j})_{i,j\in\mbz}$ with $a_{i,j}\in\mbc$ such that
\begin{itemize}
\item[(a)]$a_{i,j}=a_{i+n,j+n}$ for $i,j\in\mbz$; \item[(b)] for
every $i\in\mbz$, both sets $\{j\in\mbz\mid a_{i,j}\not=0\}$ and
$\{j\in\mbz\mid a_{j,i}\not=0\}$ are finite.
\end{itemize}
A basis for $\afgl$ can be described as $\{\afE_{i,j}\mid i,j\in\mbz\}$, where the matrix $\afE_{i,j}=(e^{i,j}_{k,l})_{k,l\in\mbz}$ is defined by
\begin{equation*}e_{k,l}^{i,j}=
\begin{cases}1&\text{if $k=i+sn,l=j+sn$ for some $s\in\mbz$,}\\
0&\text{otherwise}.\end{cases}
\end{equation*}
Let  $\afThn=\afMnn$ be the $\mbn$-span of the basis. Then $\afThn$ serves as the index set of
the PBW basis of the universal enveloping algebra $\sU(\afgl)$. Let $\bfU(\afgl)$ be the quantum enveloping algebra of $\afgl$ over $\mbq(\up)$ and let
$$\afThnpm=\{A\in\afThn\mid a_{i,i}
=0\text{ for all $i$}\}\text{ and }\afmbzn=\{(\la_i)_{i\in\mbz}\mid
\la_i\in\mbz,\,\la_i=\la_{i-n}\ \text{for}\ i\in\mbz\}.$$
Then $\afThnpm\times\afmbzn$ serves as an index set of a PBW type basis for $\bfU(\afgl)$ (see, e.g., \cite[1.4.6]{DDF}). We will construct a new basis $\{A(\bfj)\mid A\in \afThnpm,\bfj\in\afmbzn\}$ and prove the following main result.
\begin{MThm}\label{MThm} The quantum enveloping algebra $\bfU(\afgl)$ is the $\mbq(\up)$-algebra which is spanned by the basis $\{A(\bfj)\mid A\in \afThnpm,\bfj\in\afmbzn\}$ and generated by $0(\bfj)$, $S_\al(\bfl)$ and ${}^t\!S_\al(\bfl)$ for all $\bfj\in\afmbzn$ and $\al\in\mbnn$, where $S_\al=\sum_{1\leq i\leq n}\al_i\afE_{i,i+1}$ and ${}^t\!S_\al$ is the transpose of $S_\al$,  and whose
multiplication rules are given by the formulas in Proposition \ref{B(bfl,r)A(bfj,r)}(1)--(3).
\end{MThm}

We organise this paper as follows. We first recall in \S2 some preliminary results on the double Ringel--Hall algebra of a cyclic quiver and affine quantum Schur algebras. In particular, we display a PBW type basis for the former and a basis defined by certain double cosets  for the latter.  In \S3, we derive in the affine quantum Schur algebra some multiplication formulas (Theorem \ref{[B][A]}) of the basis elements by those associated with semisimple representations of the cyclic quiver. We then prove the main result via Theorem \ref{realization} in \S4. As an application of the work, we obtain certain multiplication formulas in the Ringel--Hall algebra which are not directly seen from the Hall algebra multiplication. In the Appendix, we give a proof for the length formula of the shortest representative of a double coset defined by a matrix.


\begin{Not}\label{Notaion} \rm We need the following index sets for bases of the triangular parts of $\bfU(\afgl)$.
Let 
$$\afThnp:=\{A\in\afThn\mid a_{i,j}=0\text{ for }i\geq j\}\text{ and }\afThnm=\{A\in\afThn\mid a_{i,j}=0\text{ for }i\leq j\}.$$
For $A\in\afThn$, we write
\begin{equation}\label{A^+,A^-,A^0}
A=A^\pm+A^0=A^++A^0+A^-
\end{equation}
where $A^\pm\in\afThnpm$, $A^+\in\afThnp$,
$A^-\in\afThnm$ and $A^0$ is a diagonal matrix.

Further,  for $r\geq 0$, let $\afmbnn=\{(\la_i)_{i\in\mbz}\in \afmbzn\mid \la_i\ge0\text{ for  }i\in\mbz\}$ and let
\begin{equation*}
\afThnr=\{A\in\afThn\mid\sg(A)=r\} \text{ and }
\afLanr=\{\la\in\afmbnn\mid\sg(\la)=r\}
\end{equation*}
where $\sg(A)=\sum_{1\leq i\leq n,\,
j\in\mbz}a_{i,j}$ and $\sg(\la)=\sum_{1\leq i\leq n}\la_i$. Note that $\afThnr$ is
the index set of a basis for a certain quotient algebra of $\bfU(\afgl)$---the affine quantum Schur algebras.


Moreover, we will use the standard notation for Gaussian polynomials. Thus,
let $\sZ=\mbz[\up,\up^{-1}]$, where $\up$ is an indeterminate, and let $\mbq(\up)$ be the fraction field of $\sZ$.
For integers $N,t$ with $t\geq 0$ and $\mu\in\afmbzn$ and $\la\in\afmbnn$, let
\begin{equation*}
\dleb{N\atop t}\drib=\prod\limits_{1\leq
i\leq t}\frac{\up^{2(N-i+1)}-1}{\up^{2i}-1}
\,\,\text{ and }\,\,\dleb{\mu\atop\la}\drib=\prod_{1\leq i\leq n}\dleb{\mu_i\atop\la_i}\drib.
\end{equation*}
 Then
 $\dleb{N\atop
t}\drib=\frac{\dblr{N}\dblr{N-1}\cdots \dblr{N-t+1}}{\dblr{t}^!}$, where
$\dblr{t}^!=\dblr{1}\dblr{2}\cdots \dblr{t}$ with $[\![m]\!]=\frac{\up^{2m}-1}{\up^2-1}$. We also need the symmetric Gaussian polynomials
$\leb{N\atop t}\rib=\up^{-t(N-t)}\dleb{N\atop t}\drib.$ For $\la, \la^{(1)},\ldots,\la^{(m)}\in\afmbnn$ with $\la=\la^{(1)}+\cdots+\la^{(m)}$, we also need the following notation in \ref{presentation-dbfHa}(2)(e)

$$\dleb{\la\atop\la^{(1)},\ldots,\la^{(m)}}\drib=\prod_{1\leq i\leq n}\frac{[\![\la_i]\!]^!}{[\![\la^{(1)}_i]\!]^!\cdots[\![\la^{(m)}_i]\!]^!}.$$
\end{Not}

\section{Preliminary results}

In this section, we briefly discuss the affine symmetric group and its associated Hecke algebra, the affine $\up$-Schur algebra, the double Hall algebra interpretation of affine $\mathfrak{gl}_n$ and the connections between them.

Let $\afsygr$ be the {\it affine symmetric group} consisting of all permutations
$w:\mbz\ra\mbz$ satisfying $w(i+r)=w(i)+r$ for $i\in\mbz$.
Let $W_r$ be the subgroup of
$\afsygr$, the {\it Weyl group} of affine type $A$, generated by $S=\{s_i\}_{1\leq i\leq r}$, where $s_i$  is defined by
$s_i(j)=j$ for $j\not\equiv i,i+1\mnmod r$, $s_i(j)=j-1$ for
$j\equiv i+1\mnmod r$, and $s_i(j)=j+1$ for $j\equiv i\mnmod r$.
Let $\rho$ be the permutation of $\mbz$ sending $j$ to $j+1$ for all $j\in\mbz$. We extend the length function $\ell$ on $W_r$ to $\afsygr$ by setting $\ell(\rho^mw)=\ell(w)$ for all $m\in\mbz,w\in W_r$.

The (extended) affine Hecke algebra $\afHr$ over $\sZ$ associated to
$\affSr$ is the $\sZ$-algebra which is spanned by (basis)
$\{T_w\}_{w\in\affSr}$ and generated by $T_\rho,T_{\rho^{-1}},T_s, s\in S$, and whose multiplication rules are given by the formulas, for all $s\in S$ and $w\in\fS_{\vtg,r}$,
\begin{equation*}
\aligned
T_sT_w&=\begin{cases} (\up^2-1)T_{w}+\up^2T_{sw},\quad&\text{ if }\ell(sw)<\ell(w);\\
                                         T_{sw},\quad&\text{ if } \ell(sw)=\ell(w)+1,\end{cases}\\
T_\rho T_w&=T_{\rho w}.\\
\endaligned
\end{equation*}
Let $\afbfHr=\afHr\ot_\sZ\mbq(v)$. We will discover a similar description for quantum affine $\mathfrak{gl}_n$.

For $\la\in\afLanr$, let $\fS_\la:=\fS_{(\la_1,\ldots,\la_n)}$
be the corresponding standard Young subgroup of the symmetric group $\fS_r$.
For a finite subset $X\han\affSr$, let $$T_X=\sum_{x\in
X}T_x\in\afHr\;\;\text{ and }\;\; x_\la=T_{\fS_\la}.$$
 The endomorphism algebras over $\sZ$ or $\mbq(\up)$
$$\sS_\vtg(n,r):=\End_{\afHr}\biggl
(\bop_{\la\in\La_\vtg(n,r)}x_\la\afHr\biggr)\,\text{ and }\,\afbfSr:=\End_{\afbfHr}\biggl
(\bop_{\la\in\La_\vtg(n,r)}x_\la\afbfHr\biggr)$$
are called {\it affine quantum Schur algebras} or, more specifically, {\it affine $\up$-Schur algebras} (cf. \cite{GV,Gr99,Lu99}).
Note that $\afbfSr\cong\afSr\ot_\sZ\mbq(v)$.

For $\la\in\afLanr$, denote the set of shortest representatives of right cosets of $\fS_\la$ in $\afSr$ by
$$\afmsD_\la=\{d\mid d\in\affSr,\ell(wd)=\ell(w)+\ell(d)\text{ for
$w\in\fS_\la$}\}.$$
Note that elements in $\afmsD_\la$ can be characterised as follows:
\begin{equation}\label{minimal coset representative}
\aligned
d^{-1}\in\afmsD_\la
&\iff d(\la_{0,i-1}+1)<d(\la_{0,i-1}+2)<\cdots<d(\la_{0,i-1}+\la_i),\,\forall 1\leq i\leq n,\endaligned
\end{equation}
where $\la_{0,i-1}:=\sum_{1\leq t\leq i-1}\la_t$. Moreover, $\afmsD_{\la,\mu}:=\afmsD_{\la}\cap{\afmsD_{\mu}}^{-1}$ is the set of shortest representatives of $(\fS_\la,\fS_\mu)$ double cosets.

For $\la,\mu\in\afLanr$ and $d\in\afmsD_{\la,\mu}$, define
$\phi_{\la,\mu}^d\in\sS_\vtg(n,r)$ by 
\begin{equation*}\label{def of standard basis}
\phi_{\la,\mu}^d(x_\nu h)=\de_{\mu\nu}\sum_{w\in\fS_\la
d\fS_\mu}T_wh
\end{equation*}
where $\nu\in\afLanr$ and $h\in\afHr$. Then by \cite{Gr99} the set
$\{\phi_{\la,\mu}^d\mid \la,\mu\in\afLanr,\,
d\in\afmsD_{\la,\mu}\}$ forms a $\sZ$-basis for $\sS_\vtg(n,r)$.

For $1\leq i\leq n$, $k\in\mbz$ and $\la\in\afLa(n,r)$ let $\la_{k,i-1}:=kr+\sum_{1\leq t\leq i-1}\la_t$ and
\begin{equation*}
R_{i+kn}^{\la}=\{\la_{k,i-1}+1,\la_{k,i-1}+2,\ldots,\la_{k,i-1}+\la_i=\la_{k,i}\},
\end{equation*}
By \cite[7.4]{VV99} (see also \cite[9.2]{DF09}), there is
a bijective map
\begin{equation*}
{\jmath_\vtg}:\{(\la, d,\mu)\mid
d\in\afmsD_{\la,\mu},\la,\mu\in\afLanr\}\lra\afThnr
 \end{equation*}
sending $(\la, w,\mu)$ to $A=(a_{k,l})$, where $a_{k,l}=|R_k^\la\cap wR_l^\mu|$ for all $k,l\in\mbz$.
For $A\in\afThnr$ let $e_A=\phi_{\la,\mu}^d$ where $A=\jmath_\vtg(\la,d,\mu)$. Furthermore, let
\begin{equation}\label{nbasis}
[A]=\up^{-d_A}e_{A},\quad\text{ where } \quad
d_{A}=\sum_{1\leq i\leq n\atop i\geq k,j<l}a_{i,j}a_{k,l}.
\end{equation}
Later, in \ref{eBeA} and \ref{[B][A]}, we will consider basis elements associated with matrices of the form $M=A+T-\tilde T$ for some $T,\ti T\in\afThn$. We will automatically set $e_M=0=[M]$ if one of the entry of $M$ is zero.

 For $A\in\afThnpm$ and $\dt\in\afmbzn$, define elements in
$\afbfSr$:
\begin{equation}\label{A(dt,la,r),A(dt,r)}
\begin{split}
A(\dt,r)&=\sum_{\mu\in\afLa(n,r-\sg(A))}v^{\mu\centerdot\dt}
[A+\diag(\mu)];\;\;\;\qquad(\text{cf. \cite{BLM}})\\
\end{split}
\end{equation}
where $\mu\centerdot\dt=\sum_{1\leq i\leq n}\mu_i\dt_i$. The set $\{A(\bfj,r)\}_{A\in\afThnpm, \dt\in\afmbzn}$ spans $\afbfSr$.

\vspace{.3cm}

Let $\tri$ ($n\geq 2$) be
the cyclic quiver 
with vertex set $I=\mbz/n\mbz=\{1,2,\ldots,n\}$ and arrow set
$\{i\to i+1\mid i\in I\}$. Let $\field$ be a field. For $i\in I$, let $S_i$
be the irreducible nilpotent representation of $\tri$ over $\field$ with $(S_i)_i=\field$ and $(S_i)_j=0$ for $i\neq j$.
For any $A=(a_{i,j})\in\afThnp$, let
$$M(A)=M_\field(A)=\bop_{1\leq i\leq n\atop i<j,\,j\in\mbz}a_{i,j}M^{i,j},$$
where
$M^{i,j}=M(\afE_{i,j})$ is the unique indecomposable nilpotent representation for $\tri$ of length $j-i$ with top $S_i$.
Thus, the set $\{M(A)\}_{A\in\afThnp}$ is a complete set of representatives of isomorphism classes of finite dimensional nilpotent representations of $\tri$.

The Euler form associated with the cyclic quiver $\tri$ is the
bilinear form $\lan-,-\ran$: $\afmbzn\times\afmbzn\ra\mbz$ defined by
$\lan\la,\mu\ran=\sum_{1\leq i\leq n}\la_i\mu_i-\sum_{1\leq i\leq n}\la_i\mu_{i+1}$
for $\la,\mu\in\afmbzn$.

By \cite{Ri93}, for $A,B,C\in\afThnp$,
let $\vi^{C}_{A,B}\in\mbz[\up^2]$ be the Hall polynomials such
that, for any finite field $\field_q$,
$\vi^{C}_{A,B}|_{\up^2=q}$ is equal to the number of submodules $N$ of
$M_{\field_q}(C)$ satisfying $N\cong M_{\field_q}(B)$ and $M_{\field_q}(C)/N\cong M_{\field_q}(A)$.

 By definition, the (generic) twisted {\it Ringel--Hall algebra} $\fkH_\vtg(n)$ of $\tri$ is the $\mbq(\up)$-algebra spanned by basis
$\{u_A=u_{[M(A)]}\mid A\in\afThnp\}$ whose multiplication is defined by, for all $A,B\in \afThnp$,
$$u_{A}u_{B}=\up^{\lan \bfd (A),\bfd (B)\ran}\sum_{C\in\afThnp}\vi^{C}_{A,B}u_{C},$$
where $\bfd(A)\in\mbn I$ is the dimension vector of $M(A)$.

By extending $\bfHall$ to Hopf algebras (see \ref{presentation-dbfHa}(2)(b) for multiplication)
$$\fkH_\vtg(n)^{\geq0}=\fkH_\vtg(n)\otimes \mbq(\up)[K_1^{\pm1},\ldots,K_n^{\pm1}]\text{ and }
\fkH_\vtg(n)^{\leq0}= \mbq(\up)[K_1^{\pm1},\ldots,K_n^{\pm1}]\otimes\fkH_\vtg(n)^{\rm op},$$
we define the double Ringel--Hall algebra $\dbfHa$ (cf. \cite{X97} and \cite[(2.1.3.2)]{DDF}) to be a quotient algebra of the free product $\fkH_\vtg(n)^{\geq 0}* \fkH_\vtg(n)^{\leq 0}$ via a certain skew Hopf paring $\psi:\fkH_\vtg(n)^{\geq 0}\times \fkH_\vtg(n)^{\leq 0}\ra \mbq(\up)$. In particular, there is a triangular decomposition
$$\dbfHa=\dbfHap\otimes\dbfHaz\otimes\dbfHam,$$
where $\dbfHap=\fkH_\vtg(n)$, $\dbfHaz= \mbq(\up)[K_1^{\pm1},\ldots,K_n^{\pm1}]$ and $\dbfHam=\fkH_\vtg(n)^{\rm op}$.

For $\al\in\afmbnn$ let
\begin{equation}\label{semisimple}
S_\al=\sum_{1\leq i\leq n}\al_i\afE_{i,i+1}\in\afThnp.
\end{equation}
Then $M(S_\al)=\oplus_{1\leq i\leq n}\al_iS_i$ is a semisimple representation of $\tri$.
Let $u_\al=u_{S_\al}$. 
 
 For $\al,\beta\in\afmbzn$, define a partial order on $\afmbzn$ by setting
 \begin{equation}\label{order1}
 \al\le\beta \iff \al_i\leq \beta_i\text{ for all }i\in\mbz.
 \end{equation}
 We now collects some of the results we need later, see \cite[Th.~2.5.3]{DDF} for part (1) and  \cite[2.6.7]{DDF} for part (2)(e).

\begin{Thm} \label{presentation-dbfHa}
\begin{itemize}
\item[(1)] Let $\bfU(\afgl)$ be the quantum enveloping algebra of the loop algebra of $\mathfrak{gl}_n$ defined
in \cite{Dr88} or \cite[\S2.5]{DDF}. Then
there is a Hopf algebra isomorphism $\dbfHa\cong\bfU(\afgl)$.
\item[(2)] The algebra $\dbfHa$ is the algebra over $\mbq(\up)$ which is spanned by basis
$$\{u_{A}^+K^\bfj u_{A}^-\mid A\in\afThnp,\bfj\in\afmbzn\}, \text{ where }
K\su\bfj=K_1^{j_1}\cdots K_n^{j_n},$$ and generated by
$u_\al^+$, $K_{i}^{\pm 1}$, $u_\beta^-$ $(\al,\beta\in\afThnp,\,1\leq i\leq n)$, and whose multiplication is given by
the following relations:
\begin{itemize}
\item[(a)]
$K_iK_j=K_jK_i$, $K_iK_i^{-1}=1$;
\item[(b)]
$K\su{\bfj} u_A^+=\up^{\lr{\bfd(A),\bfj}}u_A^+K\su\bfj$,
$u_A^-K\su\bfj=\up^{\lr{\bfd(A),\bfj}}K\su\bfj u_A^-$;
\item[(c)]
$u_\al^+u_A^+=\sum_{C\in\afThnp}\up^{\lan \al,\bfd(A)\ran}\vi_{S_\al,A}^C u_C^+$;
\item[(d)]
$u_\beta^-u_A^-=\sum_{C\in\afThnp}\up^{\lan \bfd(A),\beta\ran}\vi_{A,S_\beta}^C u_C^-$;
\item[(e)] For $\la,\mu\in\afmbnn$,
$u_\mu^-u_\la^+-u_\la^+ u_\mu^-=\displaystyle
\sum_{\al\not=0,\,\al\in\afmbnn\atop\al\leq\la,\,\al\leq\mu}\sum_{0\leq\ga\leq\al}
x_{\al,\ga}\ti K^{2\ga-\al} u_{\la-\al}^+u_{\mu-\al}^-,$ where
$$\aligned
x_{\al,\ga}&=\up^{\lr{\al,\la-\al}+\lr{\mu,2\ga-\al}+2\lr{\ga,\al-\ga-\la}+2\sg(\al)} 
\dleb{\la\atop\al-\ga, \la-\al,\ga}\drib\dleb{\mu\atop\al-\ga,
 \mu-\al,\ga}\drib\frac{\fka_{\al-\ga}\fka_{\la-\al} \fka_{\mu-\al}}{\fka_\la \fka_\mu}\\
 &\quad\, \times\sum_{m\geq 1,\ga^{(i)}\not=0\,\forall
i\atop\ga^{(1)}+\cdots+\ga^{(m)}=\ga}(-1)^m\up^{2\sum_{i<j}\lan\ga^{(i)},\ga^{(j)}\ran}
\fka_{\ga^{(1)}}\cdots
\fka_{\ga^{(m)}}\dleb\ga\atop\ga^{(1)},\ldots,\ga^{(m)}\drib^2
\endaligned$$
with $\fka_\beta=\displaystyle\prod_{i=1}^n\prod_{s=1}^{\beta_i}(\up^{2\beta_i}-\up^{2(s-1)})$ as defined in \cite[Lem.~3.9.1]{DDF}
and $\ti K\su\nu :=(\ti K_1)^{\nu_1}\cdots(\ti K_n)^{\nu_n}$ with $\ti K_i=K_iK_{i+1}^{-1}$ for $\nu\in\afmbzn$.
\end{itemize}
\end{itemize}
\end{Thm}

For $A\in\afThnp$, let $$\ti u_A^\pm=\up^{\dim \End(M(A))-\dim M(A)}u_A^\pm,$$
and let $\tA$ be the transpose matrix of $A$.
The relationship between $\dbfHa$ and $\afbfSr$ can be seen from the following (cf. \cite{GV,Lu99} and \cite[Prop.~7.6]{VV99}).
\begin{Thm}[{\cite[3.6.3, 3.8.1]{DDF}}] \label{zr}
For $r\geq 1$, the map $\zr:\dbfHa\ra \afbfSr$ is a surjective algebra homomorphism such that,
for all $\bfj\in \afmbzn$ and $A\in \afThnp$,
$$\zr(K^\bfj)=0(\bfj,r),\;\zr(\ti u_A^+)=A(\bfl,r),\;\;\text{and}\;\;
\zr(\ti u_A^-)=(\tA)(\bfl,r).$$
\end{Thm}

\section{Some multiplication formulas in the affine $\up$-Schur algebra}
We now derive certain useful multiplication formulas in the affine $\up$-Schur algebra and, hence, in the quantum affine $\mathfrak{gl}_n$. These formulas will be given in \ref{[B][A]} and \ref{B(bfl,r)A(bfj,r)}. They are the key to the realization of quantum affine $\frak{gl}_n$.

We need some preparation before proving \ref{[B][A]} and  \ref{B(bfl,r)A(bfj,r)}. The following result is given in \cite[3.2.3]{DDF}. 

\begin{Lem}\label{double coset}
Let $\la,\mu\in\afLanr$ and $d\in\msD_{\la,\mu}^\vtg$. Assume
$A=\jmath_\vtg(\la,d,\mu)$. Then $d^{-1}\frak S_\la
d\cap\frak S_\mu=\frak S_\nu$, where
$\nu=(\nu^{(1)},\ldots,\nu^{(n)})$ and
$\nu^{(i)}=(a_{ki})_{k\in\mbz}=(\ldots,a_{1i},\ldots,a_{ni},\ldots)$. In particular, we have
\begin{equation*}
x_\la T_d x_\mu=\displaystyle \prod_{1\leq i\leq n\atop j\in\mbz}\dblr{a_{i,j}}^! T_{\fS_\la d\fS_\mu}.
\end{equation*}
\end{Lem}

Given $A\in\afThnr$ with $A=\jmath_\vtg(\la,w,\mu)$, let $y_A=w$ be the shortest representative of the double coset $\fS_\la w\fS_\mu$.
\begin{Lem}\label{length of elements in Dla} $(1)$
For $\la\in\afLa(2,r)$ and $w\in\afmsD_\la\cap\fS_r$, we have
$\ell(w)=\sum_{1\leq i\leq\la_1}(w^{-1}(i)-i)$.

$(2)$ For any  $A\in\afThnr$, $\ell(y_A)=\displaystyle\sum_{1\leq i\leq n\atop i<k;j>l}a_{ij}a_{kl}.$
\end{Lem}
\begin{proof}Since $w\in\fS_r$, 
$w^{-1}(1)<\cdots<w^{-1}(\la_1)$ and $w^{-1}(\la_1+1)<\cdots<w^{-1}(r)$, it follows that
$$\aligned
\ell(w^{-1})&=|\{(i,j)\mid1\leq i<j\leq r,w^{-1}(i)>w^{-1}(j)\}|\\
&=|\{(i,j)\mid1\leq i\leq \la_1,\la_1+1\leq j\leq r,w^{-1}(i)>w^{-1}(j)\}|.\endaligned$$
On the other hand, for every $1\leq i\leq \la_1$, $w^{-1}(i)-i$ of the numbers $1,2,\ldots, w^{-1}(i)$ must lie in 
$\{w^{-1}(\la_1+1),\ldots,w^{-1}(r)\}$ which contribute $w^{-1}(i)-i$ inversions. Hence,
$\ell(w)=\ell(w^{-1})=(w^{-1}(1)-1)+(w^{-1}(2)-2)+\cdots+(w^{-1}(\la_1)-\la_1)$, proving part (1).

Part (2) is probably known. Since we couldn't find a proof in the literature, a proof is given in the Appendix.
\end{proof}
\begin{Lem}\label{sum}
For $a\geq 0$, $r\geq 1$ and $0\leq t\leq r$ we have
$$\sum_{X\han\{a+1,\cdots,a+r\}\atop |X|=t}v^{2\sum_{x\in X}x}=v^{2at+t(t+1)}\dleb{r\atop t}\drib.$$
\end{Lem}
\begin{proof}
We proceed by induction on $r$. The case $r=1$ is trivial. Assume now that $r>1$.
Then, by induction hypothesis,
\begin{equation*}
\begin{split}
\sum_{X\han\{a+1,\cdots,a+r\}\atop |X|=t}v^{2\sum_{x\in X}x}&=
\sum_{X\han\{a+1,\cdots,a+r-1\}\atop |X|=t}v^{2\sum_{x\in X}x}+
\sum_{Y\han\{a+1,\cdots,a+r-1\}\atop |Y|=t-1}v^{2(a+r+\sum_{x\in Y}x)}\\
&=v^{2at+t(t+1)}\dleb{r-1\atop t}\drib+v^{2(a+r)}v^{2a(t-1)+t(t-1)}\dleb{r-1\atop t-1}\drib\\
&=v^{2at+t(t+1)}\bigg(\dleb{r-1\atop t}\drib+v^{2(r-t)}\dleb{r-1\atop t-1}\drib\bigg)\\
&=v^{2at+t(t+1)}\dleb{r\atop t}\drib,
\end{split}
\end{equation*}
as desired.
\end{proof}

For $i\in\mbz$ let
$\afbse_i\in\afmbnn$ be such that
\begin{equation*}
(\afbse_i)_j=
\begin{cases}
1&\text{if $j\equiv i \mnmod n$}\\
0&\text{otherwise}.
\end{cases}
\end{equation*}

\begin{Lem}[{\cite[5.2]{Fu}}]\label{vartheta}
Let $\mu\in\afLanr$, $\bt\in\afmbnn$ and assume
$\mu\geq\bt$.

$(1)$
If $\al=\sum_{1\leq i\leq n}(\mu_i-\bt_i)\afbse_{i-1}$, $\de=(\al_0,\bt_1,\al_1,\bt_2,\cdots,\al_{n-1},\bt_n)$ and
$$\msY=\{(Y_0,Y_1,\cdots,Y_{n-1})\mid Y_i\han R_{i+1}^\mu,\,|Y_i|=\al_i,\ for\ 0\leq i\leq n-1\},$$ then there is a bijective map $$g:\afmsD_\de\cap\fS_\mu\ra\msY$$ defined by sending $w$ to $(w^{-1}X_0,w^{-1}X_1,\cdots,w^{-1}X_{n-1})$ where
$X_i=
\{\mu_{0,i}+1,\mu_{0,i}+2,\cdots,\mu_{0,i}+\al_i\},$ with $\mu_{0,i}=\sum_{1\leq s\leq i}\mu_s$ and $\mu_{0,0}=0$.

$(2)$ If $\ga=\mu-\bt$,   $\th=(\bt_1,\ga_1,\bt_2,\ga_2,\cdots,\bt_n,\ga_n)$ and
$$\msY'=\{(Y_1',Y_2',\cdots,Y_{n}')\mid Y_i'\han R_{i}^{\mu},\,|Y_i'|=\ga_i,\ for\ 1\leq i\leq n\},$$ then there is a bijective map $$g':\afmsD_{\th}\cap\fS_{\mu}\ra\msY'$$ defined by sending $w$ to $(w^{-1}X_1',w^{-1}X_2',\cdots,w^{-1}X_{n}')$ where
$X_i'=
\{\mu_{0,i-1}+\beta_i+1,\mu_{0,i-1}+\beta_i+2,\cdots,\mu_{0,i}\}.$
\end{Lem}
The injection can be seen easily by noting that
$$(\al_0+\bt_1,\al_1+\bt_2,\cdots,\al_{n-1}+\bt_n)=(\mu_1,\mu_2,\cdots,\mu_n)=(\bt_1+\ga_1,\bt_2+\ga_2,\cdots,\bt_n+\ga_n)$$
and $X_i$ (reps., $X_{i+1}'$) consists of the first $\al_i$ (reps., the last $\ga_{i+1}$) numbers in $R_{i+1}^\mu$ for all $0\leq i<n$, while the subjection is to define $w=g^{-1}(Y_0,Y_1,\ldots,Y_{n-1})$ by $w^{-1}(\mu_{0,i}+s)=k_{i,s}$
for all $0\leq i\leq n-1$ and $1\leq s\leq \mu_{i+1}$, where $Y_i=\{k_{i,1},\ldots,k_{i,\al_i}\}$, $R_{i+1}^\mu\backslash Y_i=\{k_{i,\al_i+1},k_{i,\al_i+2},\ldots,k_{i,\mu_{i+1}}\}$, and both are strictly increasing.

For $A\in\afMnz$ with $\sg(A)=r$, we denote $e_A=[A]=0\in\afSr$ if $a_{i,j}<0$ for some $i,j\in\mbz$.
There is a natural map
\begin{equation}\label{ti}
\ti\ :\afThn\ra\afThn\;\;A=(a_{i,j})\longmapsto\ti A=(\ti a_{i,j}),
\end{equation}
where $\ti a_{i,j}=a_{i-1,j}$ for all $i,j\in\mbz$.

We are now ready to establish multiplication formulas of an arbitrary basis elements $e_A$ by certain basis elements $e_B$ in the affine Schur algebra $\afSr$ over $\sZ$, where $B^+$ or ${}^t(B^-)$ defines a semisimple representation of the cyclic quiver. The significance of these formulas is the generalisation of \cite[3.5]{Lu99} (cf. \cite[3.1]{BLM}) from real roots to all roots including all imaginary roots.

\begin{Prop}\label{eBeA}
Let  $A\in\afThnr$ and $\mu=\ro(A)$.
Assume $\bt\in\afmbnn$ and $\bt\leq\mu$.
Let $\al=\sum_{1\leq i\leq n}(\mu_i-\bt_i)\afbse_{i-1}$ and $\ga=\sum_{1\leq i\leq n}(\mu_i-\bt_i)\afbse_{i}=\mu-\bt$, $B=\sum_{1\leq i\leq n}\al_i\afE_{i,i+1}+\diag(\beta)$, and $C=\sum_{1\leq i\leq n}\ga_i\afE_{i+1,i}+\diag(\beta)$. Then the
following identities hold in $\afSr$.
\begin{itemize}
\item[(1)]
$e_Be_A=\sum\limits_{T\in\afThn\atop\ro(T)=\al}v^{2\sum_{1\leq i\leq n,\,j>l} (a_{i,j}-t_{i-1,j})t_{i,l}}\prod\limits_{1\leq i\leq n\atop j\in\mbz}\dleb{a_{i,j}+t_{i,j}-t_{i-1,j}\atop t_{i,j}}\drib e_{A+T-\ti T};$
\item[(2)]
$e_Ce_A
=\sum\limits_{T\in\afThn\atop\ro(T)=\ga}
v^{2\sum_{1\leq i\leq n,\,j<l}(a_{i,j}-t_{i,j})t_{i-1,l}}\prod\limits_{1\leq i\leq n\atop j\in\mbz}\dleb{a_{i,j}-t_{i,j}+t_{i-1,j}\atop t_{i-1,j}}\drib e_{A-T+\ti T}.$
\end{itemize}
\end{Prop}
\begin{proof}
We only prove (1). The proof for (2) is entirely similar.

Let  $\la=\ro(B)$ and $\nu=\co(A)$. Assume $d_1\in\msD^\vtg_{\la,\mu}$ and $d_2\in\msD^\vtg_{\mu,\nu}$ defined by
$\jmath_\vtg(\la, d_1,\mu)=B$ and $\jmath_\vtg(\mu, d_2,\nu)=A$. Then $\la_i=\al_i+\bt_i$ and $\mu_i=\al_{i-1}+\bt_i$ for all $1\leq i\leq n$.
From  \ref{double coset} we see that
\begin{equation*}
\begin{split}
e_Be_A(x_\nu)&=T_{\frak S_\la d_1\frak S_\mu}\cdot T_{d_2}\cdot T_{\afmsD_\og\cap\frak S_\nu}\\
&=\frac{1}{\sum_{w\in\frak S_\mu}v^{2\ell(w)}}T_{\frak S_\la d_1\frak
S_\mu}\cdot T_{\frak S_\mu d_2\frak S_\nu}\\
&=\frac{1}{\sum_{w\in\frak S_\mu}v^{2\ell(w)}}\prod_{1\leq i\leq n\atop
j\in\mbz}\frac{1}{\dblr{a_{i,j}}^!}T_{\frak S_\la d_1\frak
S_\mu}\cdot T_{\frak S_\mu}\cdot T_{d_2}\cdot T_{\frak S_\nu}\\
&=\prod_{1\leq i\leq n\atop j\in\mbz}\frac{1}{\dblr{a_{i,j}}^!}T_{\frak
S_\la d_1\frak S_\mu}\cdot T_{d_2}\cdot T_{\frak S_\nu}\\
&=\prod_{1\leq i\leq n\atop j\in\mbz}\frac{1}{\dblr{a_{i,j}}^!}T_{\frak
S_\la}\cdot T_{d_1}\cdot T_{\afmsD_\de\cap\frak S_\mu}\cdot
T_{d_2}\cdot T_{\frak S_\nu}
\end{split}
\end{equation*}
where $\frak S_\og=d_2^{-1}\frak S_\mu d_2\cap\frak S_\nu$,
$\frak S_\de=d_1^{-1}\frak S_\la d_1\cap\fS_\mu$ with $\de=(\al_0,\bt_1,\al_1,\bt_2,\cdots,\al_{n-1},\bt_n)$.
By \eqref{minimal coset representative}, we have
$d_1=\rho^{-\al_0}$ (so $\ell(d_1)=0$). This together with the fact that $d_2\in\afmsD_\mu$ implies that $\ell(d_1wd_2)=\ell(d_1)+\ell(w)+\ell(d_2)=\ell(w)+\ell(d_2)$ for $w\in\afmsD_\de\cap\frak S_\mu$. Thus, we have
\begin{equation}\label{eq1 for fundentemental formulas}
e_Be_A(x_\nu)=\prod_{1\leq i\leq n\atop j\in\mbz}\frac{1}{\dblr{a_{i,j}}^!}\sum_{w\in\fS_\mu\cap\afmsD_\de}T_{\frak
S_\la}T_{d_1wd_2}T_{\frak S_\nu}
\end{equation}

For $w\in\afmsD_\de\cap\frak S_\mu$ let
$C^{(w)}=(c_{i,j}^{(w)})\in\afThnr$, where
$c_{i,j}^{(w)}=|R_i^\la\cap d_1wd_2R_j^\nu|$, and let $T^{(w)}=(t_{i,j}^{(w)})\in\afThn$,
where $t_{i,j}^{(w)}=|w^{-1}X_i\cap d_2 R_j^\nu|$ with
$X_i=
\{\mu_{0,i}+1,\mu_{0,i}+2,\cdots,\mu_{0,i}+\al_i\}$ for $1\leq i\leq n$ and $j\in\mbz$. Then $\ro(T^{(w)})=\alpha$ and 
$\co(T^{(w)})\leq \nu$.
Since $d_1^{-1}R_i^\la=\al_0+R_i^\la=(R_i^\mu\backslash X_{i-1}\cup X_i)$, we see that 
$c_{i,j}^{(w)}=|R_i^\la\cap d_1wd_2R_j^\nu|=|w^{-1}d_1^{-1}R_i^\la\cap d_2 R_j^\nu|=a_{i,j}-t_{i-1,j}^{(w)}+t_{i,j}^{(w)}$ (see the proof of \cite[5.3]{Fu}). In other words, for  all $w\in\afmsD_\de\cap\fS_\mu$,
\begin{equation}\label{eq2 for fundentemental formulas}
C^{(w)}=A+T^{(w)}-\ti T^{(w)}.
\end{equation}
In particular, $y_{C^{(w)}}\in \fS_\la d_1wd_2\fS_\nu\cap \afmsD_{\la\nu}$.

Putting $\fS_{\al_w}=y_{C^{(w)}}^{-1}\fS_\la y_{C^{(w)}}\cap\fS_\nu$,
 we have by \ref{double coset}
\[
\begin{split}
\sum_{w\in\fS_\mu\cap\afmsD_\de}T_{\frak
S_\la}T_{d_1wd_2}T_{\frak S_\nu}
&=\sum_{w\in\fS_\mu\cap\afmsD_\de,\,d_1wd_2=w'y_{C^{(w)}}w''\atop w'\in\fS_\la,\,w''\in\fS_\nu\cap\afmsD_{\al_w}}T_{\frak
S_\la}
T_{w'}T_{y_{C^{(w)}}}T_{w''}T_{\frak S_\nu}\\
&=\sum_{w\in\fS_\mu\cap\afmsD_\de,\,d_1wd_2=w'y_{C^{(w)}}w''\atop w'\in\fS_\la,\,w''\in\fS_\nu\cap\afmsD_{\al_w}}v^{2(\ell(w')+\ell(w''))}T_{\frak
S_\la} T_{y_{C^{(w)}}} T_{\frak S_\nu}\\
&=\sum_{w\in\fS_\mu\cap\afmsD_\de}v^{2(\ell(w)+\ell(d_2)-\ell(y_{C^{(w)}}))}\prod_{1\leq i\leq n,\,j\in\mbz}\dleb c_{i,j}^{(w)}\drib^! e_{C^{(w)}}(x_\nu).
\end{split}
\]

Now by \eqref{eq1 for fundentemental formulas} and \eqref{eq2 for fundentemental formulas} and noting $\ro(T^{(w)})=\al$ for   $w\in\afmsD_\de\cap\fS_\mu$,
we have
\begin{equation}\label{eq3 for fundentemental formulas}
\begin{split}
e_Be_A&=\sum_{w\in\fS_\mu\cap\afmsD_\de}v^{2(\ell(w)+\ell(d_2)-\ell(y_{C^{(w)}}))}
\prod_{1\leq i\leq n,\,j\in\mbz}\frac{\dleb c_{i,j}^{(w)}\drib^!}{\dblr{a_{i,j}}^!} e_{C^{(w)}}
\\
&=\sum_{T\in\afThn\atop\ro(T)=\al}\prod_{1\leq i\leq n\atop j\in\mbz}
\frac{\dblr{a_{i,j}-t_{i-1,j}+t_{i,j}}^!}{\dblr{a_{i,j}}^!}v^{2(\ell(d_2)-\ell(y_{A+T-\ti T}))}\bigg(\sum_{w\in\fS_\mu\cap\afmsD_\de\atop T^{(w)}=T}v^{2\ell(w)}\bigg)e_{A+T-\ti T}.
\end{split}
\end{equation}
Given $T\in\afThn$ with $\ro(T)=\al$ let
$$\msZ(T)=\{Z=(Z_{i,j})_{0\leq i\leq n-1,\,j\in\mbz}\mid |Z_{i,j}|=t_{i,j},\,Z_{i,j}\han R_{i+1}^\mu\cap d_2R_j^\nu,\ \text{for}\ 0\leq i\leq n-1,\,j\in\mbz\}.$$
If $T=T^{(w)}$ then the bijective map $g$ in \ref{vartheta} induces a  bijective map $$h_T: \{w\in\afmsD_\de\cap\fS_\mu\mid T^{(w)}=T\}\ra\msZ(T)$$ defined by sending $w$ to
$(w^{-1}(X_i)\cap d_2R_j^\nu)_{0\leq i\leq n-1,\,j\in\mbz}$.
Since for $0\leq i\leq n-1$ and $j\in\mbz$
$$ R_{i+1}^\mu\cap d_2R_j^\nu=\bigg\{\mu_{0,i}+\sum_{s\leq j-1}a_{i+1,s}+1,\mu_{0,i}+\sum_{s\leq j-1}a_{i+1,s}+2,\cdots,\mu_{0,i}+\sum_{s\leq j}a_{i+1,s}\bigg\},$$
it follows from \ref{length of elements in Dla} and the definition of $g^{-1}$ that, for $Z=(Z_{i,j})_{0\leq i\leq n-1,\,j\in\mbz}\in\msZ(T)$,
\begin{equation*}
\begin{split}
\ell(h_T^{-1}(Z))
&=\sum_{0\leq i\leq n-1}\bigg(\sum_{k\in Z_{i,j}\atop j\in\mbz}k
-\sum_{1\leq j\leq\al_i}(\mu_{0,i}+j)\bigg)\\
&=\sum_{0\leq i\leq n-1\atop j\in\mbz,\,k\in Z_{i,j}}k-\sum_{0\leq i\leq n-1}\bigg(\al_i\mu_{0,i}+\frac{\al_i(\al_i+1)}{2}\bigg).\\
\end{split}
\end{equation*}
This implies that
\begin{equation*}
\begin{split}
\sum_{w\in\fS_\mu\cap\afmsD_\de\atop T^{(w)}=T}v^{2\ell(w)}&=
\sum_{Z\in\msZ(T)}v^{2\ell(\kappa_T^{-1}(Z))}\\
&=v^{-2\sum_{0\leq i\leq n-1}(\al_i\mu_{0,i}+\al_i(\al_i+1)/2)}\prod_{0\leq i\leq n-1\atop j\in\mbz}\bigg(\sum_{Z_{i,j}\han R_{i+1}^\mu\cap d_2R_j^\nu\atop |Z_{i,j}|=t_{i,j}}v^{2\sum_{k\in Z_{i,j}}k}\bigg).
\end{split}
\end{equation*}
Consequently, by \ref{sum},
we have
\begin{equation}\label{eq4 for fundentemental formulas}
\begin{split}
\sum_{w\in\fS_\mu\cap\afmsD_\de\atop T^{(w)}=T}v^{2\ell(w)}
&=v^{2a_T}\prod_{0\leq i\leq n-1\atop j\in\mbz}\dleb{a_{i+1,j}\atop t_{i,j}}\drib=v^{2a_T}\prod_{1\leq i\leq n\atop j\in\mbz}\dleb{a_{i,j}\atop t_{i-1,j}}\drib
\end{split}
\end{equation}
where $$a_T=\sum_{0\leq i\leq n-1\atop j\in\mbz}\bigg(t_{i,j}\big(\mu_{0,i}
+\sum_{s\leq j-1}a_{i+1,s}\big)+\frac{t_{i,j}(t_{i,j}+1)}{2}\bigg)-\sum
_{0\leq i\leq n-1}\bigg(\al_i\mu_{0,i}+\frac{\al_i(\al_i+1)}{2}\bigg).$$
Since $\ro(T)=\al$ we have $\al_i=\sum_{j\in\mbz}t_{i,j}$ and
$\al_i^2=\sum_{j\in\mbz}t_{i,j}^2+2\sum_{j<l}t_{i,j}t_{i,l}$ for $0\leq i\leq n-1$. This implies that
$$a_T=\sum_{0\leq i\leq n-1\atop j,s\in\mbz,\, s<j}a_{i+1,s}t_{i,j}-\sum_{0\leq i\leq n-1\atop j,l\in\mbz,\,j<l}t_{i,j}t_{i,l}.$$
Since $d_2=y_A$ is the shortest representative in the double coset associated with $A$, by
\ref{length of elements in Dla}(2),
$$\ell(d_2)-\ell(y_{A+T-\ti T})=
\sum_{1\leq i\leq n\atop j>l}a_{i,j}t_{i,l}
-\sum_{1\leq i\leq n\atop j>l}a_{i+1,l}t_{i,j}+\sum_{1\leq i\leq n\atop j>l}
t_{i,j}t_{i,l}-\sum_{1\leq i\leq n\atop j>l}t_{i-1,j}t_{i,l}.$$
It follows that
$$a_T+\ell(d_2)-\ell(y_{A+T-\ti T})=\sum_{1\leq i\leq n\atop j>l}a_{i,j}t_{i,l}
-\sum_{1\leq i\leq n\atop j>l}t_{i-1,j}t_{i,l}.$$
Consequently, by \eqref{eq3 for fundentemental formulas}, \eqref{eq4 for fundentemental formulas} and noting $\prod_{1\leq i\leq n,\,j\in\mbz}\dblr{ t_{i,j}}^!=\prod_{1\leq i\leq n,\,j\in\mbz}\dblr{t_{i-1,j}}^!$ we have
\begin{equation*}
\begin{split}
e_Be_A
&=\sum_{T\in\afThn\atop\ro(T)=\al}v^{2(a_T+\ell(d_2)-\ell(y_{A+T-\ti T}))}\prod_{1\leq i\leq n\atop j\in\mbz}
\frac{\dblr{a_{i,j}-t_{i-1,j}+t_{i,j}}^!}{\dblr{t_{i-1,j}}^!\cdot \dblr{a_{i,j}-t_{i-1,j}}^!}e_{A+T-\ti T}\\
&=\sum\limits_{T\in\afThn\atop\ro(T)=\al}v^{2\sum_{1\leq i\leq n,\,j>l} (a_{i,j}-t_{i-1,j})t_{i,l}}\prod\limits_{1\leq i\leq n\atop j\in\mbz}\dleb{a_{i,j}+t_{i,j}-t_{i-1,j}\atop t_{i,j}}\drib e_{A+T-\ti T},
\end{split}
\end{equation*}
proving (1).
\end{proof}

Let $\bar\ :\sZ\ra\sZ$ be the ring homomorphism defined by $\bar v=v^{-1}$. We now use \ref{eBeA}  to derive the corresponding formulas for the normalised basis $\{[A]\}_{A\in\afThnr}$ defined in \eqref{nbasis}.

\begin{Prop} \label{[B][A]}
Let  $A\in\afThnr$ and $\al,\ga\in\afmbnn$.

$(1)$ For $B\in\afThnr$, if $B-\sum\limits_{1\leq i\leq n}\al_i\afE_{i,i+1}$ is a diagonal matrix and $\co(B)=\ro(A)$, then 
$$[B][A]=\sum_{T\in\afThn\atop\ro(T)=\al}v^{\bt(T,A)}\prod_{1\leq i\leq n\atop j\in\mbz}\ol{\dleb{a_{i,j}+t_{i,j}-t_{i-1,j}\atop t_{i,j}}\drib}[A+T-\ti T],$$
where $\bt(T,A)=\sum_{1\leq i\leq n,\,j\geq l}(a_{i,j}-t_{i-1,j})t_{i,l}-\sum_{1\leq i\leq n,\,j>l}(a_{i+1,j}-t_{i,j})t_{i,l}$.

$(2)$ For $C\in\afThnr$, if $C-\sum_{1\leq i\leq n}\ga_i\afE_{i+1,i}$ is a diagonal matrix and $\co(C)=\ro(A)$, then
$$[C][A]=\sum_{T\in\afThn\atop\ro(T)=\ga}v^{\bt'(T,A)}\prod_{1\leq i\leq n\atop j\in\mbz}\ol{\dleb{a_{i,j}-t_{i,j}+t_{i-1,j}\atop t_{i-1,j}}\drib}[A-T+\ti T],$$
where $\bt'(T,A)=\sum_{1\leq i\leq n,\,l\geq j}(a_{i,j}-t_{i,j})t_{i-1,l}-\sum_{1\leq i\leq n,\,l>j}(a_{i,j}-t_{i,j})t_{i,l}$.
\end{Prop}
\begin{proof}
We only prove (1). The proof for (2) is entirely similar. Note that
we have
$$\prod_{1\leq i\leq n\atop j\in\mbz}{\dleb{a_{i,j}+t_{i,j}-t_{i-1,j}\atop t_{i,j}}\drib}=v^{2\sum_{1\leq i\leq n,\, j\in\mbz}(a_{i,j}-t_{{i-1,j}})t_{i,j}}\prod_{1\leq i\leq n\atop j\in\mbz}\ol{\dleb{a_{i,j}+t_{i,j}-t_{i-1,j}\atop t_{i,j}}\drib}$$
Thus by \ref{eBeA}(1) we have
$$[B][A]=\sum_{T\in\afThn\atop\ro(T)=\al}v^{\bt(T,A)}\prod_{1\leq i\leq n\atop j\in\mbz}\ol{\dleb{a_{i,j}+t_{i,j}-t_{i-1,j}\atop t_{i,j}}\drib}[A+T-\ti T],$$
where
\begin{equation*}
\begin{split}
\bt(T,A)&=2\sum_{1\leq i\leq n,\,j>l} (a_{i,j}-t_{i-1,j})t_{i,l}+d_{A+T-\ti T}-d_A-d_B+2\sum_{1\leq i\leq n\atop j\in\mbz}(a_{i,j}-t_{{i-1,j}})t_{i,j}\\
&=2\sum_{1\leq i\leq n,\,j\geq l} (a_{i,j}-t_{i-1,j})t_{i,l}+d_{A+T-\ti T}-d_A-d_B.
\end{split}
\end{equation*}
Fix $T\in\afThn$ satisfying $\ro(T)=\al$. Then by definition we have $d_B=\sum_{1\leq i\leq n}b_{i,i}\al_i$ and
\begin{equation*}
\begin{split}
&\qquad d_{A+T-\ti T}-d_A\\
&=
\sum_{1\leq i\leq n\atop i\geq k,\,j<l}a_{i,j}(t_{k,l}-t_{k-1,l})
+\sum_{1\leq i\leq n\atop i\geq k,\,j<l}a_{k,l}(t_{i,j}-t_{i-1,j})+\sum_{1\leq i\leq n\atop i\geq k,\,j<l}(t_{i,j}-t_{i-1,j})(t_{k,l}-t_{k-1,l})\\
&=\sum_{1\leq i\leq n\atop j<l}a_{i,j}t_{i,l}
-\sum_{1\leq i\leq n\atop j<l}a_{i+1,l}t_{i,j}+\sum_{1\leq i\leq n\atop  j<l}(t_{i,j}-t_{i-1,j})t_{i,l}\\
&=\sum_{1\leq i\leq n\atop j<l}(a_{i,j}-t_{i-1,j})t_{i,l}-\sum_{1\leq i\leq n\atop j>l}(a_{i+1,j}-t_{i,j})t_{i,l}.
\end{split}
\end{equation*}
Furthermore, since $\ro(T)=\al$ and $\co(B)=\ro(A)$ we have $b_{i,i}=\sum_{j\in\mbz}(a_{i,j}-t_{i-1,j})$ and $\al_i=\sum_{l\in\mbz}t_{i,l}$ for each $i$, and hence
$$d_{A+T-\ti T}-d_A-d_B=-\sum_{1\leq i\leq n\atop j\geq l}(a_{i,j}-t_{i-1,j})t_{i,l}-\sum_{1\leq i\leq n\atop j>l}(a_{i+1,j}-t_{i,j})t_{i,l}.$$
Consequently, $\bt(T,A)=\sum_{1\leq i\leq n,\,j\geq l}(a_{i,j}-t_{i-1,j})t_{i,l}-\sum_{1\leq i\leq n,\,j>l}(a_{i+1,j}-t_{i,j})t_{i,l}$. The proof is completed.
\end{proof}

\section{Proof of the main theorem}

We now construct explicitly a subalgebra of the algebra 
$\afbfS(n):=\prod_{r\geq 1}\afbfSr$ and prove that this subalgebra is isomorphic to $\bfU(\afgl)$. Recall the elements $A(\dt,r)$ defined in \eqref{A(dt,la,r),A(dt,r)} and let
\begin{equation}\label{basis-fB}
A(\bfj)=(A(\bfj,r))_{r\geq 0}\in\afbfS(n)\quad\text{ and }\quad\fB=\{A(\bfj)\mid A\in\afThnpm,\,\bfj\in\afmbzn\}.
\end{equation}
Then $\fB$ is linearly independent by \cite[Prop.4.1(2)]{DF09}. Let $\afbfVn$ be  the $\mbq(v)$-subspace 
of $\afbfS(n)$ spanned by $\fB$. We will prove that $\afbfVn$
is a subalgebra of $\afbfS(n)$ isomorphic to $\bfU(\afgl)$ or $\dbfHa$ by \ref{presentation-dbfHa}(a). For this purpose, we need a larger spanning set containing $\fB$:
$$\widetilde\fB=\{A(\bfj, \la)\mid A\in\afThnpm,\,\bfj\in\afmbzn,\,\la\in\afmbnn\}$$
where $A(\dt,\la)=(A(\dt,\la,r))_{r\geq 0}$ with $A(\dt,\la,r)$ defined by
\begin{equation}
A(\dt,\la,r)=\sum_{\mu\in\afLa(n,r-\sg(A))}v^{\mu\centerdot\dt}
\leb{\mu\atop\la}\rib[A+\diag(\mu)]\;\quad(\text{cf. \cite[\S2]{Fu1}})
\end{equation}
Note that, for $\sigma(\la)\leq r$, $0(\bfj,\la, r)=\sum_{\mu\in\afLa(n,r),\la\leq\mu}v^{\mu\centerdot\dt}\leb{\mu\atop\la}\rib[\diag(\mu)]$

\begin{Lem}\label{spanning set of afbfVn}
The space $\afbfVn$ is spanned by the set $\widetilde \fB$. In other words, every $A(\bfj, \la)\in\afbfVn$.
\end{Lem}
\begin{proof}
Let $\afbfVnz$ be the $\mbq(v)$-subalgebra of $\afbfS(n)$ generated by $0(\pm\afbse_i)$ for $1\leq i\leq n$.
Then the set $\{0(\bfj)\mid\bfj\in\afmbzn\}$ forms a $\mbq(v)$-basis for $\afbfVnz$.
Since 
$$0(\dt,\la,r)=\prod_{1\leq i\leq n}\bigg(0(\afbse_i,r)^{j_i}
\prod_{1\leq s\leq\la_i}\frac{0(\afbse_i,r)v^{-s+1}-0(-\afbse_i,r)v^{s-1}}{v^s-v^{-s}}\bigg),\text{ where }\sigma(\la)\leq r,$$
we have
\begin{equation}\label{0span}
0(\dt,\la)=\prod_{1\leq i\leq n}\bigg(0(\afbse_i)^{j_i}
\prod_{1\leq s\leq\la_i}\frac{0(\afbse_i)v^{-s+1}-0(-\afbse_i)v^{s-1}}{v^s-v^{-s}}\bigg)\in\afbfVnz.
\end{equation}
On the other hand,  by the proof of \cite[3.4]{Fu1}), we have
$$0(\dt,\la)A(\bfl)=v^{\ro(A)\centerdot(\dt+\la)}A(\dt,\la)+
\sum_{\mu\in\mbnn,\,\bfl<\mu\leq\la}
v^{\ro(A)\centerdot(\dt+\la-\mu)}\leb{\ro(A)\atop\mu}\rib
A(\dt-\mu,\la-\mu).$$
By induction, we see that $A(\dt,\la)\in\spann\{0(\dt,\la)A(\bfl)\mid A\in\afThnpm,\,\dt\in\afmbzn,\,\la\in\afmbnn\}$.
Thus, by \eqref{0span}, $A(\dt,\la)\in\spann\{0(\bfj)A(\bfl)\mid A\in\afThnpm,\,\bfj\in\afmbzn\}$. This span equals $\afbfVn$ by \cite[(4.2.1)]{DF09} (i.e., \ref{B(bfl,r)A(bfj,r)}(1) below).
\end{proof}

 For $T=(t_{i,j})\in\afThn$ let $\de_T$ be the diagonal of $T$, i.e.,
$$\de_T=(t_{i,i})_{i\in\mbz}\in\afmbnn.$$
We now use \ref{[B][A]} to derive multiplication formulas of an arbitrary basis element by a ``semisimple generators'', which is the key to solving the realisation problem. Recall the notation in \eqref{A^+,A^-,A^0}.
\begin{Prop}\label{B(bfl,r)A(bfj,r)}
Let $\bfj\in\afmbzn$, $A\in\afThnpm$, $\al\in\afmbnn$, and $S_\al=\sum_{1\leq i\leq n}\al_i\afE_{i,i+1}$.
The following identities holds in $\afbfVn$:
\begin{itemize}
\item[(1)] $0(\bfj')A(\bfj)=\up^{\bfj'\centerdot\ro(A)}A(\bfj'+\bfj)$ and $A(\bfj)0(\bfj')=\up^{\bfj'\centerdot\co(A)}A(\bfj'+\bfj)$ {\rm(\cite[(4.2.1)]{DF09})}.
\item[(2)] $\displaystyle S_\al(\bfl)A(\bfj)=
\sum_{T\in\afThn\atop\ro(T)=\al}v^{f_{A,T}}\prod_{1\leq i\leq n\atop j\in\mbz,\,j\not=i}
\ol{\dleb{a_{i,j}+t_{i,j}-t_{i-1,j}\atop t_{i,j}}\drib}(A+T^\pm-\ti T^\pm)(\bfj_T,\de_T)$,\\
where $\bfj_T=\bfj+\sum_{1\leq i\leq n}(\sum_{j<i}(t_{i,j}-t_{i-1,j}))\afbse_i$ and\vspace{-1ex}
\begin{equation*}
\begin{split}
f_{A,T}&=\sum_{1\leq i\leq n\atop j\geq l,\,j\not=i}
a_{i,j}t_{i,l}-\sum_{1\leq i\leq n\atop j>l,\,j\not=i+1}a_{i+1,j}t_{i,l}
-\sum_{1\leq i\leq n\atop j\geq l,\,j\not=i}t_{i-1,j}t_{i,l}+\sum_{1\leq i\leq n\atop j>l,\,j\not=i,\,j\not=i+1}t_{i,j}t_{i,l}\\
&\qquad+\sum_{1\leq i\leq n\atop j<i+1}t_{i,j}t_{i+1,i+1}+\sum_{1\leq i\leq n}j_i(t_{i-1,i}-t_{i,i});
\end{split}
\end{equation*}
\item[(3)] $\displaystyle {}^t\!S_\al(\bfl)A(\bfj)=
\sum_{T\in\afThn\atop\ro(T)=\al}v^{f_{A,T}'}\prod_{1\leq i\leq n\atop j\in\mbz,\,j\not=i}
\ol{\dleb{a_{i,j}-t_{i,j}+t_{i-1,j}\atop t_{i-1,j}}\drib}(A-T^\pm+\ti T^\pm)(\bfj'_T,\de_{\ti T})$,\\
where $\bfj'_T=\bfj+\sum_{1\leq i\leq n}(\sum_{j>i}(t_{i-1,j}-t_{i,j}))\afbse_i$ and
\begin{equation*}
\begin{split}
f_{A,T}'&=\sum_{1\leq i\leq n\atop l\geq j,\,j\not=i}
a_{i,j}t_{i-1,l}-\sum_{1\leq i\leq n\atop l>j,\,j\not=i}a_{i,j}t_{i,l}
-\sum_{1\leq i\leq n\atop j\geq l,\,l\not=i}t_{i-1,j}t_{i,l}+\sum_{1\leq i\leq n\atop j>l,\,l\not=i,\,l\not=i+1}t_{i,j}t_{i,l}\\
&\qquad+\sum_{1\leq i\leq n\atop i<j}t_{i,j}t_{i-1,i}+\sum_{1\leq i\leq n}j_i(t_{i,i}-t_{i-1,i}).
\end{split}
\end{equation*}
\end{itemize}
In particular, $\afbfVn$ is closed under the multiplication by the ``generators'' $0(\bfj), S_\al(\bfl), {}^t\!S_\al(\bfl)$ for all $\bfj\in\mbn$ and $\al\in\mbnn$.
\end{Prop}
\begin{proof} We only prove (2). If $r<\sigma(A)$, the $r$-th components of both sides are 0. Assume now
$r\geq\sigma(A)$. By \ref{[B][A]} and noting the fact that, for $X,Y\in\afThnr$, 
$[X][Y]\neq0\implies \co(X)=\ro(Y)$  the $r$-th component of $S_\al(\bfl)A(\bfj)$ becomes
\begin{equation*}
\begin{split}
S_\al(\bfl,r)A(\bfj,r)&=\sum_{\ga\in\afLa(n,r-\sg(A))}v^{\ga\centerdot\bfj}
\bigg[S_\al+\diag\bigg(\ga+\ro(A)-\sum_{1\leq i\leq n}\al_i\afbse_{i+1}\bigg)\bigg][A+\diag(\ga)]\\
&=\sum_{T\in\afThn\atop\ro(T)=\al}
\prod_{1\leq i\leq n\atop j\in\mbz,\,j\not=i}\ol{\dleb{a_{i,j}+t_{i,j}-t_{i-1,j}\atop t_{i,j}}\drib}x_T
\end{split}
\end{equation*}
where $$x_T=
\sum_{\ga\in\afLa(n,r-\sg(A))}
v^{\ga\centerdot\bfj+\bt(T,A+\diag(\ga))}\ol{\dleb\ga+\de_T-\de_{\ti T}\atop\de_T)\drib}
[A+\diag(\ga)+T-\ti T].
$$
Let $A+\diag(\ga)=(a_{i,j}^\ga)$. Then $a_{i,j}^\ga=a_{i,j}$ for $i\not=j$ and $a_{i,i}^\ga=\ga_i$. Let $\nu=\ga+\de_T-\de_{\ti T}$. Then
\begin{equation*}
\begin{split}
\bt(T,A+\diag(\ga))&=\sum_{1\leq i\leq n\atop j\geq l}(a_{i,j}^\ga-t_{i-1,j})t_{i,l}-\sum_{1\leq i\leq n\atop j>l}(a_{i+1,j}^\ga-t_{i,j})t_{i,l}\\
&=\sum_{1\leq i\leq n\atop j\geq l,\,j\not=i}(a_{i,j}-t_{i-1,j})t_{i,l}+\sum_{1\leq i\leq n\atop i\geq l}(\nu_i-t_{i,i})t_{i,l}-\sum_{1\leq i\leq n\atop j>l,\,j\not=i+1}(a_{i+1,j}-t_{i,j})t_{i,l}\\
&\qquad -\sum_{1\leq i\leq n\atop i+1>l}(\nu_{i+1}-t_{i+1,i+1})t_{i,l}\\
&=\bt_{A,T}+\bt_{\nu,T},
\end{split}
\end{equation*}
where $\bt_{\nu,T}=\sum_{1\leq i\leq n,\,i\geq l}\nu_it_{i,l}-\sum_{1\leq i\leq n,\,i+1>l}\nu_{i+1}t_{i,l}$ and
\begin{equation*}
\begin{split}
\bt_{A,T}&=\sum_{1\leq i\leq n\atop j\geq l,\,j\not=i}(a_{i,j}-t_{i-1,j})t_{i,l}-\sum_{1\leq i\leq n\atop j>l,\,j\not=i+1}a_{i+1,j}t_{i,l}+\sum_{1\leq i\leq n\atop j>l,\,j\not=i,i+1}t_{i,j}t_{i,l}\\
&\qquad -\sum_{1\leq i\leq n}t_{i,i}^2+\sum_{1\leq i\leq n\atop i+1>l}t_{i+1,i+1}t_{i,l}.
\end{split}
\end{equation*}
Clearly, we have  $\ol{\dbbl{\nu\atop\de_T}\dbbr}=v^{\de_T\centerdot(\de_T-\nu)}\bbl{\nu\atop
\de_T}\bbr$, $\bt_{A,T}+\de_T\centerdot\de_T+\bfj\centerdot(\de_{\ti T}-\de_T)=f_{A,T}$ and $\bt_{\nu,T}+\nu\centerdot(\bfj-\de_T)=\nu\centerdot\bfj_T$. This implies that
\begin{equation*}
\begin{split}
x_T&=v^{\bt_{A,T}+\de_T\centerdot\de_T+\bfj\centerdot(\de_{\ti T}-\de_T)}
\sum_{\nu\in\afLa(n,r-\sg(A+T^\pm-\ti T^\pm))}v^{\bt_{\nu,T}+\nu\centerdot(\bfj-\de_T)}\leb{\nu\atop\de_T}\rib\\
&\qquad\times
[A+T^\pm-\ti T^\pm+\diag(\nu)]\\
&=v^{f_{A,T}}(A+T^\pm-\ti T^\pm)(\bfj_T,\de_T,r),
\end{split}
\end{equation*}
proving (2).
\end{proof}

For $A,B\in\afThnpm$, define the ordering on $\afThn$ by setting
\begin{equation}\label{order2}
A\preceq B\iff \sum\limits_{s\leq i,t\geq j}a_{s,t}\leq \sum\limits_{s\leq i,t\geq j}b_{s,t},\,\forall i<j,\text{ and }
\sum\limits_{s\geq i,t\leq j}a_{s,t}\leq \sum\limits_{s\geq i,t\leq j}a_{s,t},\,\forall i>j.
\end{equation}

\begin{Prop}\label{triangular formula in A(bfj)} With the notation in \eqref{A^+,A^-,A^0}
 we have, for any $A\in\afThnpm$ and $\bfj\in\afmbnn$,
\[
A^+(\bfl)0(\bfj)A^-(\bfl)=v^{\bfj\centerdot(\co(A^+)+\ro(A^-))}A(\bfj)+ \sum_{B\in\afThnpm
\atop B\p A,\,\bfj'\in\afmbnn}f_{A,\bfj}^{B,\bfj'}B(\bfj'),
\]
 where $f_{A,\bfj}^{B,\bfj'}\in\mbq(v)$.
\end{Prop}
\begin{proof}
Let $\dbfHap$ be the subspace of $\dbfHa$ spanned by the elements $u_A^+$ for $A\in\afThnp$.
According to \cite[6.2]{DDX}, the algebra $\dbfHap$ is generated by the elements
$\ti u_{S_\al}^+$ for $\al\in\afmbnn$,  where $S_\al$ is defined as in \eqref{semisimple}. This together with \ref{zr} implies that
$A^+(\bfl)$ can be written as a linear combination of monomials in $S_\al(\bfl)$. Thus, by \ref{B(bfl,r)A(bfj,r)} and \ref{spanning set of afbfVn}, we conclude that there exist $f_{A,\bfj}^{B,\bfj'}\in\mbq(v)$ (independent of $r$) such that
\begin{equation}\label{eq1 triangular formula in A(bfj)}
A^+(\bfl)0(\bfj)A^-(\bfl)=\sum_{B\in\afThnpm
\atop  \bfj'\in\afmbnn}f_{A,\bfj}^{B,\bfj'}B(\bfj').
\end{equation}
On the other hand, by the triangular relation given in \cite[3.7.3]{DDF}, we have
$$
A^+(\bfl,r)0(\bfj,r)A^-(\bfl,r)=v^{\bfj\centerdot(\co(A^+)+\ro(A^-))}A(\bfj,r)+f
$$
where $f$ is a $\mbq(v)$-combination of $[B]$ with $B\in\afThnr$ and $B\p A$. Combining this with \eqref{eq1 triangular formula in A(bfj)} proves the assertion.
\end{proof}

The maps $\zeta_r$ given in \ref{zr} induce an algebra homomorphism
$$\zeta=\prod_{r\geq 1}\zeta_r:\dbfHa\lra\afbfS(n)=\prod_{r\geq 1}\afbfSr.$$
We now prove the conjecture formulated in \cite[5.5(2)]{DF09}.
\begin{Thm}\label{realization}
The $\mbq(v)$-space $\afbfVn$ is a subalgebra of $\afbfS(n)$
with $\mbq(v)$-basis ${\mathfrak B}$.
Moreover, the map  $\zeta$  is injective
and induces a $\mbq(v)$-algebra isomorphism $\dbfHa\overset\zeta\cong\afbfVn$.
\end{Thm}
\begin{proof}
According to \cite[4.1]{DF09}, the set $\frak B$ forms a $\mbq(v)$-basis for $\afbfVn$. This together with \ref{triangular formula in A(bfj)} implies that
the set $\{A^+(\bfl)0(\bfj)A^-(\bfl)\mid A\in\afThnpm,\,\bfj\in\afmbzn\}$ forms another basis for $\afbfVn$. 
Note that, by \ref{presentation-dbfHa}(2), the set $\{\ti u_{A^+}^+K^\bfj\ti u_{{}^t\!(A^-)}^-\mid A\in\afThnpm,\,\bfj\in\afmbzn\}$ forms a $\mbq(v)$-basis for $\dbfHa$. Furthermore, by \ref{zr},
$$\zeta(\ti u_{A^+}^+K^\bfj\ti u_{{}^t\!(A^-)}^-)=A^+(\bfl)0(\bfj)A^-(\bfl).$$
 Thus, $\zeta$ takes a basis for $\dbfHa$ onto the basis for $\afbfVn$. It follows that $\zeta$ is injective and $\zeta(\dbfHa)=\afbfVn$. The proof is completed.
\end{proof}

Now, the Main Theorem \ref{MThm} follows immediately.

We end the paper with an application to the (untwisted) Ringel--Hall algebra of a cyclic quiver.

Let $\afbfVnp$ be the subalgebra of $\afbfVn$ spanned by $A(\bfl)$ for all $A\in\afThnp$. Then, by \ref{zr},
the map sending $\ti u_A$ to $A(\bfl)$ is an algebra isomorphism from $\dbfHap=\bfHall$ to $\afbfVnp$. In particular, the formula \ref{B(bfl,r)A(bfj,r)}(2) gives the following multiplication formula in the Ringel--Hall algebra $\Hall$
over $\sZ$:
$$\ti u_\al\ti u_A=
\sum_{T\in\afThnp\atop\ro(T)=\al}v^{f_{A,T}}\prod_{1\leq i\leq n\atop j\in\mbz,\,j\not=i}
\ol{\dleb{a_{i,j}+t_{i,j}-t_{i-1,j}\atop t_{i,j}}\drib}\ti u_{A+T-\ti T^+}$$
where 
$$f_{A,T}=\displaystyle\sum_{1\leq i\leq n\atop j\geq l,\,j\not=i}
a_{i,j}t_{i,l}-\sum_{1\leq i\leq n\atop j>l,\,j\not=i+1}a_{i+1,j}t_{i,l}
-\sum_{1\leq i\leq n\atop j\geq l,\,j\not=i}t_{i-1,j}t_{i,l}+\sum_{1\leq i\leq n\atop j>l,\,j\not=i,\,j\not=i+1}t_{i,j}t_{i,l}.$$
Untwisting the multiplication for $\Hall$ yields the following.

 \begin{Thm} The Ringel--Hall algebra $\Hall^\diamond$ is the algebra over the polynomial ring $\mbz[\bsq]$ ($\bsq=\up^2$) which is spanned by the basis $\{u_A\mid A\in\afThnp\}$ and generated by $\{u_\al\mid\al\in\afmbnn\}$ and whose (untwisted) multiplication is given by the formulas: for any $A\in\afThnp$ and $\al\in\afmbnn$,
$$ u_\al \diamond u_A =
\sum_{T\in\afThnp\atop\ro(T)=\al} \bsq^{\sum_{1\leq i \leq n,\,l< j}(a_{i,j}t_{i,l}-t_{i,j}t_{i+1,l})} \prod_{1\leq i\leq n\atop j\in\mbz,\,j\not=i}
{\dleb{a_{i,j}+t_{i,j}-t_{i-1,j}\atop t_{i,j}}\drib}u_{A+T-\ti T^+}$$
\end{Thm}
In other words, if the Hall polynomial $\vi_{S_\al,A}^B$ is nonzero, then there exists
$T=(t_{i,j})\in\afThnp$ with $\ro(T)=\al$ such that $B=A+T-\ti T^+$ and 
$$\vi_{S_\al,A}^{A+T-\ti T^+}=\bsq^{\sum_{1\leq i \leq n,\,l< j}(a_{i,j}t_{i,l}-t_{i,j}t_{i+1,l})} \prod_{1\leq i\leq n\atop j\in\mbz,\,j\not=i}
{\dleb{a_{i,j}+t_{i,j}-t_{i-1,j}\atop t_{i,j}}\drib}.$$
Note that, when $\al$ defines a simple module or $A$ defines another semisimple module, the formula coincides with the formulas given in \cite[Th.~5.4.1]{DDF} (built on \cite[Th.~4.2]{DF09})  and \cite[Cor.~1.5]{DDM}.

\def\Inv{{\text{Inv}}}

\section{Appendix --- Proof of  Lemma \ref{length of elements in Dla}(2)}

For $w\in\affSr$ and $t\in\mbz$, let 
\begin{equation}\label{Inv(w)}
\aligned
\text{Inv}(w,t)&=\{(i,j)\in\mbz^2\mid1+t\leq i\leq r+t,\ i<j,\ w(i)>w(j)\}\\
\Inv(w)&=\Inv(w,0).\endaligned
\end{equation}
Then the number of inversions $\ell'(w):=|\text{Inv}(w,t)|$ is clearly independent of $t$.

\begin{Prop}\label{inversion}
If $w=ys$ with $y,w\in\affSr$ and $s\in S$ satisfies $\ell(w)=\ell(y)+1$, then $\ell'(w)=\ell'(y)+1$.
\end{Prop}
\begin{proof}Suppose $s=s_{i_0}$ for some $1\leq i_0\leq r$. By the hypothesis and \cite[4.2.3]{Shi86}, we have $y(i_0)<y(i_0+1)$. Fix $t\in\{0,1\}$ with
$i_0, i_0+1\in[1+t,r+t]$.
Let
$\sW=\text{Inv}(w,t)\text{ and }\sY=\text{Inv}(y,t).$ We want to prove that $|\sW|=|\sY|+1$. For $j\in\mbz$ and $i\in[i_0+1,i_0+r]$, let 
$$c(i_0,i,j):=|\{i_0,i_0+1\}\cap\{i,\bar j\}|,$$ where $\bar j$ denote the unique integer in $[1+t,r+t]$ such that $j\equiv\bar j\mod r$. For each $x\in\{0,1,2\}$, let 
$$\sW_x=\{(i,j)\in\sW\mid c(i_0,i,j)=x\}\text{ and } \sY_x=\{(i,j)\in\sY\mid c(i_0,i,j)=x\}.$$
Then we have disjoint unions $\sW=\sW_0\cup\sW_1\cup\sW_2$ and $\sY=\sY_0\cup\sY_1\cup\sY_2$.

For $j\in\mbz$ and $i\in[i_0+1,i_0+r]$, if $c(i_0,i,j)=0$, then $w(i)=y(i)$ and $w(j)=y(j)$. This implies $\sW_0=\sY_0$. Hence, $|\sW_0|=|\sY_0|$.
Since $y(i_0)<y(i_0+1)$, it follows that
$$\sW_2=\{(i_0,i_0+1+kr)\mid k\in\mbz_{\geq0}, y(i_0+1)>y(i_0)+kr\},$$
while
$$\sY_2=\{(i_0+1,i_0+kr)\mid k\in\mbz_{>0}, y(i_0+1)>y(i_0)+kr\}.$$
Hence, $|\sW_2|=|\sY_2|+1$. It remains to prove that $|\sW_1|=|\sY_1|$.
In this case, we have
$$\sW_1=\sW_{1,(i_0,\bullet)}\cup \sW_{1,(i_0+1,\bullet)}\cup\sW_{1,(\bullet,i_0)}\cup\sW_{1,(\bullet,i_0+1)},$$
where $\sW_{1,(i_0,\bullet)}=\{(i,j)\in\sW_1\mid i=i_0\}$, etc. Define $\sY_{1,(\bullet,\bullet)}$ similarly to get a similar partition for $\sY_1$. Then
the condition
$y(i_0)<y(i_0+1)$ implies the following
$$\aligned
\sY_{1,(i_0,\bullet)}\subseteq \sW_{1,(i_0,\bullet)},\qquad\sY_{1,(i_0+1,\bullet)}\supseteq \sW_{1,(i_0+1,\bullet)},\\
\sY_{1,(\bullet,i_0)}\supseteq \sW_{1,(\bullet,i_0)},\qquad\sY_{1,(\bullet,i_0+1)}\subseteq \sW_{1,(\bullet,i_0+1)}.
\endaligned
$$
But, since
$$\aligned
 \sW_{1,(i_0,\bullet)}\backslash \sY_{1,(i_0,\bullet)}&=\{(i_0,j)\mid i_0<j,j\in\mbz,\bar j\not\in\{i_0,i_0+1\},y(i_0)<y(j)<y(i_0+1)\}\text{ and}\\
\sY_{1,(i_0+1,\bullet)}\backslash  \sW_{1,(i_0+1,\bullet)}&=\{(i_0+1,j)\mid i_0+1<j,j\in\mbz,\bar j\not\in\{i_0,i_0+1\},y(i_0)<y(j)<y(i_0+1)\},\\
 \endaligned
 $$
it follows that  $|\sW_{1,(i_0,\bullet)}|+|\sW_{1,(i_0+1,\bullet)}|=|\sY_{1,(i_0,\bullet)}|+|\sY_{1,(i_0+1,\bullet)}|$. Similarly, one proves that $|\sW_{1,(\bullet,i_0)}|+|\sW_{1,(\bullet,i_0+1)}|=|\sY_{1,(\bullet,i_0+1)}|+|\sY_{1,(\bullet,i_0)}|$. This completes the proof.
\end{proof}

The following result given in \cite[(3.2.1.1)]{DDF} without proof follows immediately.

\begin{Coro}\label{combinatorial description of length}
For $w\in\affSr$ we have
$\ell(w)=\ell'(w)$.
\end{Coro}

We now generalise the construction for the shortest representatives of double cosets of the symmetric group 
\cite[\S3]{Du} to the affine case.
For $A\in\afThnr$ with $\la=\ro(A)$ and $\mu=\co(A)$, define a pseudo matrix $A^-$ as follows: the entry $a_{i,j}$ is replaced by the sequence
$$\ul{c}_{i,j}=\ul{c}_{i,j}(A)=\bigg(\la_{k_0,i_0-1}+\sum_{t\leq j-1}a_{i,t}+1,\cdots,\la_{k_0,i_0-1}+\sum_{t\leq j-1}a_{i,t}+(a_{i,j}-1) ,\la_{k_0,i_0-1}+\sum_{t\leq j}a_{i,t}\bigg)$$
where $i=i_0+k_0n,j\in\mbz$ with $1\leq i_0\leq n$ and $k_0\in\mbz$. We define $\ti y_{A}\in\affSr$ by 
$$\ti y_{A}(i+kr)=a_i+kr,\text{ for all }1\leq i\leq r,k\in\mbz,$$
where $(a_1,a_2,\cdots,a_r)$ is the sequence obtained by reading the numbers in column 1 inside the subsequences from left to right and from top to bottom, and then in column 2, etc., and then in column $n$. In other words, it is the sequence obtained by ignoring 0's from $((\ul{c}_{k,1})_{k\in\mbz},
(\ul{c}_{k,2})_{k\in\mbz},\cdots,(\ul{c}_{k,n})_{k\in\mbz})$ with $(\ul{c}_{k,i})_{k\in\mbz}=(\cdots,\ul{c}_{1,i},\ul{c}_{2,i},\cdots,\ul{c}_{n,i},\cdots)$.

We are ready to prove Lemma \ref{length of elements in Dla}(2).
\begin{Prop}
Let  $A\in\afThnr$. Then $\ti y_A$ is the shortest representative of the double coset defined by $A$, i.e., $y_A=\ti y_A$, and
$$\ell(y_A)=\sum_{1\leq i\leq n\atop i<k;j>l}a_{ij}a_{kl}.$$
\end{Prop}
\begin{proof} Let $\la=ro(A)$ and $\mu=co(A)$. For $w\in\fS_\la y_A\fS_\mu$, let
$$\sN=\bin_{1\leq l\leq n\atop i<k;j>l}(R_i^\la\cap w(R_j^\mu))\times(R_k^\la\cap w(R_l^\mu))\text{ and }N=|\sN|=\sum_{1\leq i\leq n\atop i<k;j>l}a_{ij}a_{kl}.$$
 Then there is a injective map $\vi_w$ defined as follows:
\begin{equation*}
\vi_{w}:\sN\lra\Inv(w),\;\;(c,d)\longmapsto (w^{-1}(d),w^{-1}(c)).
\end{equation*}
Hence by \ref{combinatorial description of length} we have
$N\leq \ell(w)$. In particular, we have $\ell(y_A)\geq N$.

For $i,j\in\mbz$ let $C_{ij}(A)$ denote the set of members of the sequence $\ul{c}_{ij}(A)$. By definition we have,
for $i\in\mbz$ and $1\leq j\leq n$,
$$R_i^\la=\bin_{l\in\mbz}C_{il}(A),\quad\text{and}\quad \ti y{_A}(R_j^\mu)=\bin_{k\in\mbz}C_{kj}(A).$$
 It is easy to see that $C_{ij}(A)+tr=C_{i+tn,j+tn}(A)$ for $i,j,t\in\mbz$. Hence,
for $j=j_0+tn$ with $1\leq j_0\leq n$ and $t\in\mbz$, 
$$\ti y{_A}(R_j^\mu)=tr+ w{_A}(R_{j_0}^\mu)=tr+\bin_{k\in\mbz}C_{kj_0}(A)=\bin_{k\in\mbz}C_{k+tn,j_0+tn}(A)=\bin_{k\in\mbz}C_{k,j}(A).$$ Thus, we have $C_{ij}(A)=R_i^\la\cap \ti y{_A}(R_j^\mu)$ for $i,j\in\mbz$ and so $a_{ij}=|R_i^\la\cap \ti y{_A}(R_j^\mu)|$ for $i,j\in\mbz$.
This implies that $\ti y_A\in\fS_\la y_A\fS_\mu$ and, hence, $\ell(\ti y_A)\geq\ell(y_A)\geq N$. 
Observe  that $\ti y{_A}(C_{ji}(\tA))=C_{ij}(A)$ for $i,j\in\mbz$, where $\tA$ is the transpose matrix of $A$.

We now prove that $\ell(\ti y_A)=N$ by showing that $\vi_{\ti y{_A}}$ is surjective. Let $(a,b)\in\Inv(\ti y{_A})$. Since $\mbz=\bin_{s,t\in\mbz}C_{s,t}(A)$, there exist $i,j,k,l\in\mbz$ such that
$\ti y{_A}(a)\in C_{kl}(A)$ and $\ti y{_A}(b)\in C_{ij}(A)$. Since $1\leq a\leq r$  we have $1\leq l\leq n$. Since $\ti y{_A}(a)>\ti y{_A}(b)$ we have
either $i<k$ or $i=k$ and $j<l$. On the other hand, the conditions $a\in \ti y{_A}^{-1}(C_{kl}(A))=C_{lk}(\tA)$, $b\in \ti y{_A}^{-1}(C_{ij}(A))=C_{ji}(\tA)$ and $a<b$ imply either $l<j$ or $l=j$ and $k<i$. Hence, we must have $i<k$ and $l<j$. Therefore, the map $\vi_{\ti y{_A}}$ is bijective, proving $\ell(\ti y{_A})=N$.
\end{proof}


\begin{thebibliography}{99}
\bibitem{BLM}
A.~A. Beilinson, G. Lusztig and R. MacPherson, {\em A geometric
setting for the quantum deformation of $GL_n$}, Duke Math.J. {\bf
61} (1990), 655--677.

\bibitem{Bri} T. Bridgeland, {\em Quantum groups via Hall algebras of complexes}, Ann. Math. {\bf 177} (2013), 739--759.

\bibitem{DDF}
B. Deng, J. Du and Q. Fu,
{\em A double Hall algebra approach to affine quantum Schur--Weyl theory}, London
Mathematical Society Lecture Note Series, {\bf 401}, Cambridge University Press, 2012.

\bibitem{DDM}
B. Deng, J. Du and A. Mah,
{\em Presenting degenerate Ringel--Hall algebras of cyclic quivers}, J. Pure Appl. Algebra, {\bf214} (2010), 1787--1799.

\bibitem{DDX}
B. Deng, J. Du and J. Xiao, {\em Generic extensions and canonical
bases for cyclic quivers}, Can. J. Math. {\bf 59} (2007),
1260--1283.

\bibitem{Dr88} V.~G. Drinfeld, {\em A new realization of Yangians and quantized
affine alegbras}, Soviet Math. Dokl. {\bf 32} (1988), 212--216.

\bibitem{Du} J. Du, {\em Cells in certain sets of matrices}, T\^ohoku Math. J. {\bf 48}(1996), 417--427.

\bibitem{DF09}
J. Du and Q. Fu,
{\em A modified BLM approach to quantum affine
$\frak{gl}_n$}, Math. Z. {\bf 266} (2010), 747--781.

\bibitem{FM} E. Frenkel and E. Mukhin, {\em The Hopf algebra Rep $U_q(\mathfrak{gl}_\infty)$,} Sel. Math., New Ser. {\bf 8} (2002), 537--635.

\bibitem{Fu}
Q. Fu, \textit{BLM realization for $\sU_\mbz(\h{\frak{gl}}_n)$}, preprint, arXiv:1204.3142.


\bibitem{Fu1}
Q. Fu, \textit{BLM realization for the integral form of quantum $\frak{gl}_n$},
preprint, arXiv:1211.0671.

\bibitem{GV}
V. Ginzburg and E. Vasserot, {\em Langlands reciprocity for affine quantum groups of type $A_n$}, Internat. Math. Res. Notices 1993, 67--85.


\bibitem{Gr99}
R. M. Green, {\em The affine $q$-Schur algebra}, J. Algebra {\bf 215} (1999),  379--411.




\bibitem{Lu99}
G. Lusztig, {\em Aperiodicity in quantum affine $\frak{gl}_n$},
Asian J. Math. {\bf 3} (1999),  147--177.

\bibitem{R90} C.~M. Ringel, {\em Hall algebras and quantum groups},
Invent. Math. {\bf 101} (1990), 583--592.

\bibitem{R932} C.~M. Ringel, {\em Hall algebras revisited},
Israel Mathematical Conference Proceedings, Vol. {\bf 7} (1993), 171--176.

\bibitem{Ri93}
C.~M. Ringel, {\em The composition algebra of a
cyclic quiver}, Proc. London Math. Soc. {\bf 66} (1993), 507--537.

\bibitem{Shi86}
J. Y. Shi, {\it The Kazhdan-Lusztig cells in certain affine Weyl groups}, Lecture Notes in Mathematics, {\bf 1179}. Springer-Verlag, Berlin, 1986.

\bibitem{VV99}
M. Varagnolo and E. Vasserot, {\em On the
decomposition matrices of the quantized Schur algebra}, Duke Math.
J. {\bf 100} (1999), 267--297.

\bibitem{X97}
J. Xiao, {\em Drinfeld double and Ringel-Green theory of Hall algebras},
J. Algebra {\bf 190} (1997), 100--144.

\bibitem{SY}
S. Yanagida, {\em A note on Bridgeland's Hall algebra of two-periodic complexes}, preprint, arXiv:1207.0905.
\end{thebibliography}
\end{document}